\documentclass[12pt]{article}
\usepackage{amssymb,amsmath,latexsym,amsthm}
\usepackage{mathrsfs}
\usepackage{url}
\usepackage{enumerate}
\usepackage{ifpdf}
\ifpdf
  \usepackage[pdftex]{hyperref}
\else
  \usepackage[dvips]{hyperref}
\fi

\title{Joint convergence along different subsequences of the signed cubic variation of fractional Brownian motion II}
\author{David Nualart\\University of Kansas
  \and
  Jason Swanson\\University of Central Florida
  }
\date{February 22, 2013}

\DeclareMathOperator{\sgn}{sgn}%
\usepackage{color}

\begin{document}

\newtheorem{thm}{Theorem}[section]
\newtheorem{conj}[thm]{Conjecture}
\newtheorem{lemma}[thm]{Lemma}
\theoremstyle{definition}
\newtheorem{prob}[thm]{Problem}
\numberwithin{equation}{section}

\def\al{\alpha}
\def\be{\beta}
\def\ga{\gamma}
\def\Ga{\Gamma}
\def\de{\delta}
\def\De{\Delta}
\def\ep{\varepsilon}
\def\eps{\varepsilon}
\def\ze{\zeta}
\def\th{\theta}
\def\ka{\kappa}
\def\la{\lambda}
\def\La{\Lambda}
\def\vpi{\varpi}
\def\si{\sigma}
\def\Si{\Sigma}
\def\ph{\varphi}
\def\om{\omega}
\def\Om{\Omega}

\def\wt{\widetilde}
\def\wh{\widehat}
\def\ol{\overline}
\def\ds{\displaystyle}

\def\nab{\nabla}
\def\pa{\partial}
\def\To{\Rightarrow}
\def\eqd{\overset{d}{=}}
\def\emp{\emptyset}

\def\pf{\noindent{\bf Proof.} }
\def\qed{\hfill $\Box$}

\providecommand{\flr}[1]{\left\lfloor{#1}\right\rfloor}
\providecommand{\ceil}[1]{\left\lceil{#1}\right\rceil}
\providecommand{\ang}[1]{\left\langle{#1}\right\rangle}

\def\bA{\mathbb{A}}
\def\bB{\mathbb{B}}
\def\bC{\mathbb{C}}
\def\bD{\mathbb{D}}
\def\bE{\mathbb{E}}
\def\bF{\mathbb{F}}
\def\bG{\mathbb{G}}
\def\bH{\mathbb{H}}
\def\bI{\mathbb{I}}
\def\bJ{\mathbb{J}}
\def\bK{\mathbb{K}}
\def\bL{\mathbb{L}}
\def\bM{\mathbb{M}}
\def\bN{\mathbb{N}}
\def\bO{\mathbb{O}}
\def\bP{\mathbb{P}}
\def\bQ{\mathbb{Q}}
\def\bR{\mathbb{R}}
\def\bS{\mathbb{S}}
\def\bT{\mathbb{T}}
\def\bU{\mathbb{U}}
\def\bV{\mathbb{V}}
\def\bW{\mathbb{W}}
\def\bX{\mathbb{X}}
\def\bY{\mathbb{Y}}
\def\bZ{\mathbb{Z}}

\def\bfA{{\bf A}}
\def\bfB{{\bf B}}
\def\bfC{{\bf C}}
\def\bfD{{\bf D}}
\def\bfE{{\bf E}}
\def\bfF{{\bf F}}
\def\bfG{{\bf G}}
\def\bfH{{\bf H}}
\def\bfI{{\bf I}}
\def\bfJ{{\bf J}}
\def\bfK{{\bf K}}
\def\bfL{{\bf L}}
\def\bfM{{\bf M}}
\def\bfN{{\bf N}}
\def\bfO{{\bf O}}
\def\bfP{{\bf P}}
\def\bfQ{{\bf Q}}
\def\bfR{{\bf R}}
\def\bfS{{\bf S}}
\def\bfT{{\bf T}}
\def\bfU{{\bf U}}
\def\bfV{{\bf V}}
\def\bfW{{\bf W}}
\def\bfX{{\bf X}}
\def\bfY{{\bf Y}}
\def\bfZ{{\bf Z}}

\def\cA{\mathcal{A}}
\def\cB{\mathcal{B}}
\def\cC{\mathcal{C}}
\def\cD{\mathcal{D}}
\def\cE{\mathcal{E}}
\def\cF{\mathcal{F}}
\def\cG{\mathcal{G}}
\def\cH{\mathcal{H}}
\def\cI{\mathcal{I}}
\def\cJ{\mathcal{J}}
\def\cK{\mathcal{K}}
\def\cL{\mathcal{L}}
\def\cM{\mathcal{M}}
\def\cN{\mathcal{N}}
\def\cO{\mathcal{O}}
\def\cP{\mathcal{P}}
\def\cQ{\mathcal{Q}}
\def\cR{\mathcal{R}}
\def\cS{\mathcal{S}}
\def\cT{\mathcal{T}}
\def\cU{\mathcal{U}}
\def\cV{\mathcal{V}}
\def\cW{\mathcal{W}}
\def\cX{\mathcal{X}}
\def\cY{\mathcal{Y}}
\def\cZ{\mathcal{Z}}

\def\sA{\mathscr{A}}
\def\sB{\mathscr{B}}
\def\sC{\mathscr{C}}
\def\sD{\mathscr{D}}
\def\sE{\mathscr{E}}
\def\sF{\mathscr{F}}
\def\sG{\mathscr{G}}
\def\sH{\mathscr{H}}
\def\sI{\mathscr{I}}
\def\sJ{\mathscr{J}}
\def\sK{\mathscr{K}}
\def\sL{\mathscr{L}}
\def\sM{\mathscr{M}}
\def\sN{\mathscr{N}}
\def\sO{\mathscr{O}}
\def\sP{\mathscr{P}}
\def\sQ{\mathscr{Q}}
\def\sR{\mathscr{R}}
\def\sS{\mathscr{S}}
\def\sT{\mathscr{T}}
\def\sU{\mathscr{U}}
\def\sV{\mathscr{V}}
\def\sW{\mathscr{W}}
\def\sX{\mathscr{X}}
\def\sY{\mathscr{Y}}
\def\sZ{\mathscr{Z}}

\def\fA{\mathfrak{A}}
\def\fB{\mathfrak{B}}
\def\fC{\mathfrak{C}}
\def\fD{\mathfrak{D}}
\def\fE{\mathfrak{E}}
\def\fF{\mathfrak{F}}
\def\fG{\mathfrak{G}}
\def\fH{\mathfrak{H}}
\def\fI{\mathfrak{I}}
\def\fJ{\mathfrak{J}}
\def\fK{\mathfrak{K}}
\def\fL{\mathfrak{L}}
\def\fM{\mathfrak{M}}
\def\fN{\mathfrak{N}}
\def\fO{\mathfrak{O}}
\def\fP{\mathfrak{P}}
\def\fQ{\mathfrak{Q}}
\def\fR{\mathfrak{R}}
\def\fS{\mathfrak{S}}
\def\fT{\mathfrak{T}}
\def\fU{\mathfrak{U}}
\def\fV{\mathfrak{V}}
\def\fW{\mathfrak{W}}
\def\fX{\mathfrak{X}}
\def\fY{\mathfrak{Y}}
\def\fZ{\mathfrak{Z}}

\maketitle

\begin{abstract}

The purpose of this paper is to provide a complete description the convergence in distribution of two subsequences of the signed cubic variation of the fractional Brownian motion with Hurst parameter $H = 1/6$.

\medskip

\noindent{\bf AMS subject classifications:} Primary 60G22; secondary 60F17.

\noindent{\bf Keywords and phrases:} Fractional Brownian motion, cubic variation, convergence in law.

\end{abstract}
 
\section{Introduction}
Suppose that $B=\{B(t), t\ge 0\}$ is a fractional Brownian motion with Hurst parameter $H=\frac 1 6$. Let $\flr{x}$ denote the greatest integer less than or equal to $x$.  In \cite{NualartOrtiz2008}, Nualart and Ortiz-Latorre proved that the sequence of sums,
  \[
  W_n(t) = \sum_{j=1}^{\flr{nt}} (B(j/n) - B((j-1)/n))^{3},
  \]
converges in law to a  Brownian motion $W=\{W(t), t\ge 0 \}$, with variance   $\kappa^2t$  given by
  \begin{equation*}
  \ka^2 = \frac34\sum_{m\in\bZ}
    (|m + 1|^{1/3} + |m - 1|^{1/3} - 2|m|^{1/3})^3.
  \end{equation*}
The process $W$ is related to the signed cubic variation of  $B$. A detailed analysis of this process has been recently developed by Swanson in \cite{Swanson2011a}, considering this variation as a class of sequences of processes. 

In \cite{BNS}, Burdzy, Nualart and Swanson studied the convergence in distribution of the sequence of two-dimensional processes $\{(W_{a_n}(t), W_{b_n}(t))\}$, where $\{a_n\}_{n=1}^\infty$ and $\{b_n\}_{n=1}^\infty$ are two strictly increasing sequences of natural numbers converging to infinity.  A basic assumption for the results of \cite{BNS} and also for the results of this paper is that   $L_n\to L\in[0,\infty]$, where $L_n=b_n/a_n$. By \cite[Corollary 3.6]{BNS}, if $L\in\{0,\infty\}$, then $W_{a_n}$ and $W_{b_n}$ converge to independent Brownian motions. We will therefore assume that $L\in(0,\infty)$.

The function $f_L(x) = \sum_{m\in\bZ}f_{m,L}(x)$, where
  \begin{equation} \label{ef}
  f_{m,L}(x) = (|x - m + 1|^{1/3} + |x - m - L|^{1/3}
    - |x - m|^{1/3} - |x - m + 1 - L|^{1/3})^3,
  \end{equation}
plays a fundamental role in the analysis of the convergence in distribution of $\{(W_{a_n}(t), W_{b_n}(t))\}$. Under some conditions, the limit of this sequence    is a two-dimensional Gaussian process $X^\rho$, independent of $B$, whose components are Brownian motions with variance $\kappa^2t$, and with covariance $\int_0 ^t \rho(s)\,ds$ for some function $\rho$. In terms of  the function  $\rho\in C[0,\infty)$, the process $X^\rho$ can be expressed as
 \begin{equation} \label{main1b}
  X^\rho(t) = \int_0^t \si(s)\,d\mathbf{W}(s),
  \end{equation}
  where  $\si$ is given by
  \begin{equation}\label{main1a}
  \si(t) = \ka\begin{pmatrix}
    \sqrt{1 - |\ka^{-2}\rho(t)|^2} &\ka^{-2}\rho(t)\\
    0 &1
    \end{pmatrix},
  \end{equation}
and  $\mathbf{W}=(W^1, W^2)$ is a standard, 2-dimensional Brownian motion.
More specifically, the main result of \cite{BNS} is the  following theorem, which is obtained  using  the central limit theorem for multiple stochastic integrals proved by Peccati and Tudor in \cite{PeccatiTudor2005} (see also \cite{NuPe}).

\begin{thm}\label{C:main3}
  Let $I=\{n:L_n=L\}$ and $c_n=\gcd(a_n,b_n)$.  
Then $(B,W_{a_n},W_{b_n}) \To (B,X^\rho)$ in  the Skorohod space $D_{\bR^3}[0,\infty)$ as $n\to\infty$, in the following cases:
  \begin{enumerate}[(i)]
  \item  The set  $I^c$ is finite (which implies $L\in\bQ$). In this case, if $L=p/q$, where $p,q\in\bN$ are relatively prime, then for all $t\ge 0$,
    \begin{equation*} 
    \rho(t) = \frac3{4p}\sum_{j=1}^q f_L(j/q).
    \end{equation*}
  \item   There exists $k\in\bN$ such that $b_n=k\mod a_n$ for all $n$. In this case, for all $t\ge0$,
      \[
    \rho(t) = \frac3{4L}f_L(kt).
    \]
  \item The set $I$ is finite and $c_n\to\infty$. In this case, for all $t\ge0$,
    \[
    \rho(t) = \frac3{4L}\int_0^1 f_L(x)\,dx.
    \]
 
  \end{enumerate}
\end{thm}

This  type of result  was motivated  by the relationship between higher  signed  variations of fractional Brownian motions and the change of variable  formulas in distribution  for stochastic integrals with respect to these processes that have appeared recently in the literature  (see \cite{BS,NR,NoReSw}).  

Theorem \ref{C:main3} covers many simple and interesting pairs of sequences, and helps to tell a surprising story about the asymptotic correlation between the sequences, $\{W_{a_n}\}$ and $\{W_{b_n}\}$, both of which are converging to a Brownian motion. For example, by Theorem \ref{C:main3}(i), we may conclude that the asymptotic correlation of $W_n(t)$ and $W_{2n}(t)$ is a constant that does not depend on $t$, and whose numerical value is approximately 0.201. Likewise, Theorem \ref{C:main3}(iii) shows that the asymptotic correlation of $W_{n^2}(t)$ and $W_{n(n+1)}(t)$ is not dependent on $t$ and is approximately 0.102. Perhaps more surprisingly, Theorem \ref{C:main3}(ii) shows that the asymptotic correlation of $W_n(t)$ and $W_{n+1}(t)$ \textit{does} depend on $t$. Numerical calculations suggest that the correlation varies greatly with $t$, converging to 1 as $t\downarrow 0$, and being as low as about 0.075 for $t=0.8$.

Nonetheless, there are many simple and interesting pairs of sequences that are \textit{not} covered by Theorem \ref{C:main3}. For example, the sequences $a_n=n^2$ and $b_n=(n+1)^2$ are not covered; nor are the sequences $a_n=2n$ and $b_n=3n+1$. Additionally, many sequences whose ratios converge to an irrational number are not covered by this theorem. 

The purpose of this paper is to provide a complete description of the asymptotic behavior of $W_{a_n}(t)$ and $W_{b_n}(t)$ for all sequences $\{a_n\}$ and $\{b_n\}$. We will  show that the asymptotic correlation depends only on $L=\lim L_n$ when $L$ is irrational; and when $L$ is rational, it depends also on $\lim a_n|L_n-L|$. In the next section we state and prove this result and provide some remarks and examples.

\section{Main result}

Let $X^\rho$ the two-dimensional process defined in  (\ref{main1b}). Recall that $f_L(x) =\sum_{m\in \bZ} f_{m,L}(x)$, where $f_{m,L}$ is the function defined in (\ref{ef}). By \cite[Lemma 2.6]{BNS}, the series defining $f_L$ converges uniformly on $[0,1]$. Also note that $f_L$ is periodic with period 1. We first need the following technical result.

\begin{lemma} \label{lem1}
Let $L=p/q$, where $p,q \in \bN$ are relatively prime numbers. Then, for any $x\in \bR$ and $\eta =1, \dots, q$ we have $f_L(\eta L-x) =f_L(\wt\eta L+x)$, where $\wt\eta=q-\eta+1$.
\end{lemma}

\begin{proof}
For any $m\in \bZ$ set  $\wt m =-m+1+p$. Then
  \begin{align*}
  f_{m,L}(\eta L - x) &= \Bigg(
    \bigg|\frac{\eta p}q - x - m + 1\bigg|^{1/3}
    + \bigg|\frac{\eta p}q - x - m - \frac pq\Bigg|^{1/3}\\
  &\qquad\qquad
    - \bigg|\frac{\eta p}q - x - m \bigg|^{1/3}
    - \bigg|\frac{\eta p}q - x - m + 1 - \frac pq\bigg|^{1/3}
  \Bigg)^3\\
  &= \Bigg(
    \bigg|-\frac{\eta p}q + x - \wt m + p\bigg|^{1/3}
    + \bigg|-\frac{\eta p}q + x - \wt m + 1 + p + \frac pq\bigg|^{1/3}\\
  &\qquad\qquad
    - \bigg|-\frac{\eta p}q + x - \wt m + 1 + p\bigg|^{1/3}
    - \bigg|-\frac{\eta p}q + x - \wt m + p +\frac pq\bigg|^{1/3}
    \Bigg)^3.
  \end{align*}
Notice that $\wt \eta L = p+\frac pq -\frac {\eta p}q$. Therefore,
  \begin{align*}
  f_{m,L}(\eta L - x) &= \big(
    |\wt\eta L + x - \wt m - L|^{1/3}
    + |\wt\eta L + x - \wt m + 1|^{1/3}\\
  &\qquad\qquad
    - |\wt\eta L + x - \wt m - L + 1|^{1/3}
    - |\wt\eta L + x - \wt m|^{1/3}
    \big)^3\\
  &= f_{\wt m,L}(\wt\eta L + x).
  \end{align*}
As a consequence,
  \[
  f_L(\eta L - x) = \sum_{m\in\bZ} f_{m,L}(\eta L - x)
    = \sum_{m\in\bZ} f_{\wt m,L}(\wt\eta L + x)
    = \sum_{\wt m\in\bZ} f_{\wt m,L}(\wt\eta L + x)
    = f_L(\wt\eta L + x),
  \]
which completes the proof.
\end{proof}

The next result is the main theorem of this paper. Together with the cases $L=0$ and $L=\infty$, covered in \cite[Corollary 3.6]{BNS}, this theorem gives a complete description of all subsequential limits of $(W_{a_n},W_{b_n})$ for any pair of subsequences of $\{W_n\}$.

\begin{thm}\label{t2}
Let $\{a_n\}_{n=1}^\infty$ and $\{b_n\}_{n=1}^\infty$ be strictly increasing sequences in $\bN$. Let $L_n=b_n/a_n$ and suppose $L_n\to L\in(0,\infty)$. Let $\de_n=L_n-L$.   Then, $(B,W_{a_n},W_{b_n}) \To (B,X^\rho)$ in $D_{\bR^3}[0,\infty)$ as $n\to\infty$, in the following cases:
\begin{enumerate}[(i)]
\item $L\in\bQ$ and $a_n|\de_n|\to k\in[0,\infty)$. In this case, if we write $L=p/q$, where $p,q\in\bN$ are relatively prime, then, for all $t\ge 0$,
  \[
  \rho (t) = \frac3{4p}\sum_{j=1}^q
    f_L\left({\frac jq + kt}\right).
  \]
\item $L\in\bQ$ and $a_n|\de_n|\to\infty$, or $L\notin\bQ$. In this case, for all $t\ge0$,
  \[
  \rho (t) = \frac3{4L}\int_0^1 f_L(x)\,dx.
  \]
\end{enumerate}
\end{thm}

Note that between the two parts of this theorem, there is, at least formally, a sort of continuity in $k$. For fixed $q$, since $f_L$ is periodic with period 1, we have
  \[
  \int_0^t\bigg[\frac3{4p}\sum_{j=1}^q
    f_L\left({\frac jq + ks}\right)\bigg]\,ds
    \to \int_0^t\bigg[\frac3{4L}\int_0^1 f_L(x)\,dx\bigg]\,ds,
  \]
as $k\to\infty$.

To elaborate on the conditions in the two parts of this theorem and their connections to Theorem \ref{C:main3}, first note that if $L$ is rational and $L_n\ne L$, then
  \begin{equation}\label{num_thry}
  a_n|\de_n| = \left|{\frac{b_nq - a_np}{q}}\right| \ge \frac1q,
  \end{equation}
since the numerator is a nonzero integer. It follows that when $L\in\bQ$, we have $a_n|\de_n|\to0$ if and only if $L_n=L$ for all but finitely many $n$. Therefore, Theorem \ref{t2}(i) with $k=0$ is equivalent to Theorem \ref{C:main3}(i).

Next, if $L\in\bQ$, $L_n\ne L$ for all but finitely many $n$, and $c_n=\gcd(a_n,b_n)\to\infty$, then \eqref{num_thry} shows that for $n$ sufficiently large, $a_n|\de_n|\ge c_n/q\to\infty$. Hence, Theorem \ref{C:main3}(iii) is a special case of Theorem \ref{t2}(ii).

Lastly, to see that Theorem \ref{C:main3}(ii) is a special case of  Theorem \ref{t2}(i), suppose there exists $k\in\bN$ such that $b_n=k\mod a_n$ for all $n$. Then $b_n=\nu_na_n+k$ for some integers $\nu_n$. Thus, $L_n=\nu_n+k/a_n$. Letting $n\to\infty$ shows that $L\in\bN$ and $\nu_n=L$ for all but finitely many $n$. We therefore have $a_n|\de_n| = |b_n - a_nL| = k$, for large enough $n$. In this case, using $p=L$ and $q=1$ and the fact that $f_L$ is periodic with period 1, we find that the function $\rho$ in Theorem \ref{t2}(i) agrees with the function $\rho$ in Theorem \ref{C:main3}(ii).

Before giving the formal proof of Theorem \ref{t2} we would like to explain the main ideas in comparison with the proof of  Theorem \ref{C:main3}.  Let $\{x\}=x-\flr{x}$. In \cite{BNS}, it is shown that the covariance between the components of the limit process $X^\rho$ is given by
  \[
  \int_0^t \rho(s)\,ds = \frac3{4L}\sum_{m\in\bZ}\lim_{n\to\infty}
    \frac1{a_n}\sum_{j=1}^{\flr{a_nt}}f_{m,L}(\{jL_n\}),
  \]
provided the above limits exist for each $m\in\bZ$. The principal challenge in analyzing these limits has been that the above summands, $f_{m,L}(\{jL_n\})$, could not be replaced by $f_{m,L}(\{jL\})$. This is because, although $L_n$ is close to $L$ for large $n$, $\{jL_n\}$ is not uniformly close to $\{jL\}$ as $j$ ranges from 1 to $\flr{a_nt}$. In \cite{BNS}, we studied these limits via the decomposition
  \[
  \sum_{j=1}^{\flr{a_nt}}f_{m,L}(\{jL_n\})
    = \al_n\sum_{j=0}^{q_n-1} f_{m,L}(\{jL_n\})
    + \sum_{j=1}^{r_n} f_{m,L}(\{jL_n\}).
  \]
Here, $b_n/a_n=p_n/q_n$, where $p_n$ and $q_n$ are relatively prime, and $\flr{a_nt}=\al_nq_n + r_n$ with $\al_n\in\bZ$ and $0\le r_n<q_n$.

To prove Theorem \ref{t2} in the case that $L\in\bQ$, we  use a different decomposition. Let $L=p/q$, where $p,q\in\bN$ are relatively prime. We then write
  \[
  \sum_{j=1}^{\flr{a_nt}}f_{m,L}(\{jL_n\})
    \approx \sum_{\eta=1}^q\sum_{i=0}^{\al_n-1}
    f_{m,L}(\{(iq + \eta)L_n\}).
  \]
In this case, since $q$ is fixed and finite, we are able to use the approximation
  \[
  \sum_{j=1}^{\flr{a_nt}}f_{m,L}(\{jL_n\})
    \approx \sum_{\eta=1}^q\sum_{i=0}^{\al_n-1}
    f_{m,L}(\{iqL_n + \eta L\}).
  \]
Since $qL=p$, we have $iqL_n=ip+iq\de_n$. Thus, we have
  \[
  \sum_{j=1}^{\flr{a_nt}}f_{m,L}(\{jL_n\})
    \approx \sum_{\eta=1}^q\sum_{i=0}^{\al_n-1}
    f_{m,L}(\{iq\de_n + \eta L\}).
  \]
Using a Riemann-sum argument, we will show  that for each fixed $\eta$,
  \[
  \sum_{i=0}^{\al_n-1} f_{m,L}(\{iq\de_n + \eta L\})
    \approx \frac1{q\de_n}\int_0^{a_n\de_nt} f_{m,L}(\{x + \eta L\})\,dx,
  \]
giving
  \[
  \int_0^t \rho(s)\,ds = \frac3{4L}\sum_{m\in\bZ}
    \frac1q\sum_{\eta=1}^q\lim_{n\to\infty}
    \frac1{a_n\de_n}\int_0^{a_n\de_nt} f_{m,L}(\{x + \eta L\})\,dx.
  \]
We then prove the theorem case-by-case, depending on the asymptotic behavior of the sequence $a_n\de_n$. Note that the actual analysis in the proof is made somewhat more delicate by the fact that $\de_n$ may be negative.

For the case $L\notin\bQ$,   the proof will be done by adapting  the method of proof of the equidistribution theorem based on Fourier series expansions. This theorem  says that for any interval $I\subset [0,1)$,
  \[
  \lim_{n\to\infty} \frac1n |\{k: \{kL\} \in I, 1 \le k \le n\}| = |I|,
  \]
and a simple proof can be found in \cite[Theorem 1.8]{Wa}.

\begin{proof}[Proof of Theorem \ref{t2}]
Let
  \[
  S_n(t) = E[W_{a_n}(t) W_{b_n}(t)].
  \]
By \cite[Theorem 3.1 and Lemma 3.5]{BNS}, it will suffice to show that
  \begin{equation}\label{t2.pregoal}
  S_n(t) \to \int_0^t \rho(s)\,ds,
  \end{equation}
for each $t\ge0$.

Fix $t\ge0$. Since $W_n(t)=0$ if $\flr{nt}=0$, we may assume $t>0$ and $n$ is sufficiently large so that $\flr{a_nt}>0$ and $\flr{b_nt}>0$. Recall that $\{x\}=x-\flr{x}$, and let $\wh{f}_{m,L}(x)= f_{m,L} (\{x\})$, where $f_{m,L}$ is the function introduced in $(\ref{ef})$.

In the reference \cite{BNS} it is proved (see \cite[(3.18), (3.20), and Remark 3.3]{BNS}) that
  \begin{equation}\label{t2.goal}
  \lim_{n\rightarrow \infty} S_n(t) =\frac3{4L}\sum_{m\in\bZ}
    \lim_{n\to\infty}\wt\be(m,n),
  \end{equation}
where
  \[
  \wt\be(m,n) = \frac1{a_n}\sum_{j=1}^{\flr{a_nt}}
    \wh f_{m,L}(jL_n),
  \]
provided that, for each fixed $m\in \bZ$, the limit $\lim_{n\to\infty}\wt\be(m,n)$ exists. The proof will now be done in several steps.

\medskip

\noindent\textbf{Step 1.} Assume $L\in\bQ$ and $a_n|\de_n| \to k\in(0,\infty]$. Let us write $L=p/q$, where $p$ and $q$ are relatively prime. Choose $n_0$ such that for all $n\ge n_0$, we have $\lfloor a_n t \rfloor >q$. For each $n\ge n_0$, write $\flr{a_nt} = \al_n q+r_n$, where $\al_n\in\bN$ and $0\le r_n< q$. Since $a_n\to\infty$ and $\wh f_{m,L}$ is bounded, it follows that
  \begin{align*}
  \lim_{n\to\infty}\wt\be(m,n)
    &= \lim_{n\to\infty}\frac1{a_n}\sum_{j=1}^{\al_n q}
    \wh f_{m,L}(jL_n)\\
  &= \lim_{n\to\infty}\frac1{a_n}
    \sum_{\eta=1}^q\sum_{i=0}^{\al_n-1}
    \wh f_{m,L}((iq + \eta)L_n)\\
  &= \lim_{n\to\infty}\frac1{a_n}
    \sum_{\eta=1}^q\sum_{i=0}^{\al_n-1}
    \wh f_{m,L}(ip + \eta L + (iq + \eta)\de_n)\\
  &= \sum_{\eta=1}^q\bigg(\lim_{n\to\infty}\frac1{a_n}
    \sum_{i=0}^{\al_n-1}\wh f_{m,L}(\eta L + \sgn(\de_n)x_i)\bigg),
  \end{align*}
where $x_i =(iq+\eta)|\de_n|$. Our assumption that $a_n |\de_n |\to k \in (0, \infty]$ implies that there exists $n_1\ge n_0$ such that $\de_n \ne0$ for all $n\ge n_1$.  Set $\Delta x=x_{i+1} -x_i=q|\de_n|$. Then
  \[
  \lim_{n\to\infty}\wt\be(m,n)
    = \frac 1q\sum_{\eta =1}^q\bigg(\lim_{n\to\infty}
    \frac 1{a_n|\de_n|}\sum_{i=0} ^{\al_n-1}
    \wh f_{m,L} (\eta L + \sgn(\de_n)x_i) \Delta x\bigg).
  \]
Let $\ep>0$ be arbitrary. Since $f_{m,L}$ is continuous, we may find $n_2 \ge n_1$ such that for all $n\ge n_2$,
  \[
  \sup_{\substack{|x-y|\le\De x\\ x,y\in[0,1]}}
    |f_{m,L}(x) - f_{m,L}(y)| < \ep.
  \]
Note that if $\flr{x}=\flr{y}$, then $\{x\}-\{y\}=x-y$. Thus,
  \begin{equation}\label{unifsub}
  \sup_{\substack{|x-y|\le\De x\\ \flr{x}=\flr{y}}}
    |\wh f_{m,L}(x) - \wh f_{m,L}(y)| < \ep,
  \end{equation}
for all $n\ge n_2$. Let
  \[
  J_n = \{0\le i < \al_n:
    \flr{\eta L + \sgn(\de_n)x_i} = \flr{\eta L + \sgn(\de_n)x_{i+1}}\}.
  \]
Note that if $i\in J_n$ and $x\in[x_i,x_{i+1}]$, then
  \[
  \flr{\eta L + \sgn(\de_n)x} = \flr{\eta L + \sgn(\de_n)x_i}.
  \]
Thus, using \eqref{unifsub}, we obtain
  \[
  \bigg|\wh f_{m,L} (\eta L + \sgn(\de_n)x_i) \Delta x
    - \int_{x_i}^{x_{i+1}}
    \wh f_{m,L} (\eta L + \sgn(\de_n)x)\,dx\bigg|
    \le \ep\De x = \ep q|\de_n|,
  \]
for all $i\in J_n$ and $n\ge n_2$. Also, since $\wh f_{m,L}$ is bounded, there is a constant $M$ such that
  \[
  \bigg|\wh f_{m,L} (\eta L + \sgn(\de_n)x_i) \Delta x
    - \int_{x_i}^{x_{i+1}}
    \wh f_{m,L} (\eta L + \sgn(\de_n)x)\,dx\bigg|
    \le M\De x = M q|\de_n|,
  \]
for all $i\notin J_n$ and $n\ge n_2$. Therefore,
  \[
  \sum_{i=0}^{\al_n-1}\wh f_{m,L}(\eta L + \sgn(\de_n)x_i)\De x
    = \int_{x_0}^{x_{\al_n}}\wh f_{m,L}(\eta L + \sgn(\de_n)x)\,dx
    + R_n,
  \]
where
  \[
  |R_n| \le (\ep|J| + M(\al_n - |J|))q|\de_n|.
  \]
Note that $\al_n-|J|$ is the number of times that the monotonic sequence $\{\eta L + \sgn(\de_n)x_i\}_{i=0}^{\al_n}$ crosses an integer. Thus, $\al_n - |J| \le |x_{\al_n} - x_0| + 1=\al_nq|\de_n|+1$. Combined with $|J|\le\al_n$ and $\al_n\le a_nt/q$, we have
  \[
  |R_n| \le \ep a_n|\de_n|t + Mqa_n|\de_n|^2t + Mq|\de_n|.
  \]
Hence, since $a_n\to\infty$, we have
  \[
  \limsup_{n\to\infty}\frac{|R_n|}{a_n|\de_n|}
    \le \ep t.
  \]
Since $\ep$ was arbitrary, it follows that
  \[
  \lim_{n\to\infty}\wt\be(m,n)
    = \frac 1q\sum_{\eta =1}^q\bigg(\lim_{n\to\infty}
    \frac 1{a_n|\de_n|}\int_{x_0}^{x_{\al_n}}
    \wh f_{m,L}(\eta L + \sgn(\de_n)x)\,dx\bigg).
  \]
Now, note that $x_0=\eta|\de_n|$ and
  \[
  x_{\al_n} = (\al_nq + \eta)|\de_n|
    = (\flr{a_nt} - r_n + \eta)|\de_n|.
  \]
Since $|\eta-r_n|\le q$, we have $|x_{\al_n}-a_n|\de_n|t|\le(q+1)|\de_n|$. Thus, since $a_n\to\infty$ and $\wh f_{m,L}$ is bounded, we have
  \begin{equation}\label{t2.1}
  \lim_{n\to\infty}\wt\be(m,n)
    = \frac 1q\sum_{\eta =1}^q\bigg(\lim_{n\to\infty}
    \frac 1{a_n|\de_n|}\int_0^{a_n|\de_n|t}
    \wh f_{m,L}(\eta L + \sgn(\de_n)x)\,dx\bigg).
  \end{equation}
  
\medskip

\noindent\textbf{Step 2.} Assume $L\in\bQ$ and $a_n|\de_n| \to \infty$. Then, taking into account that the function $\wh f_{m,L}$ has period one, we can write
  \begin{multline*}
  \int_0^{a_n|\de_n|t}\wh f_{m,L}(\eta L + \sgn(\de_n)x)\,dx\\
    = \flr{a_n|\de_n|t}\int_0^1\wh f_{m,L}(x)\,dx
    + \int_0^{a_n|\de_n|t-\flr{a_n|\de_n|t}}
    \wh f_{m,L}(\eta L+\sgn(\de_n)x)\,dx.
  \end{multline*}
From \eqref{t2.1} and the fact that $\wh f_{m,L}$ is bounded, we then obtain
  \[
  \lim_{n\to\infty}\wt\be(m,n) = t\int_0^1\wh f_{m,L}(x)\,dx.
  \]
By \eqref{t2.goal} and the fact that $f_L=\sum_{m\in\bZ}f_{m,L}$ is periodic with period 1, this gives
  \[
  \lim_{n\to\infty} S_n(t) = \frac{3t}{4L}\int_0^1 f_L(x)\,dx.
  \]
In light of \eqref{t2.pregoal}, this completes half the proof of Theorem \ref{t2}(ii). To complete the proof of Theorem \ref{t2}(ii), it remains only to consider the case $L\notin\bQ$, and this will be done in the final step of this proof.

\medskip

\noindent\textbf{Step 3.} Assume $L\in\bQ$, $a_n|\de_n| \to k\in(0,\infty)$, and $\de_n >0$ for all $n$. From \eqref{t2.1}, we have
  \[
  \lim_{n\to\infty}\wt\be(m,n)
    = \frac 1q\sum_{\eta =1}^q\bigg(
    \frac 1k\int_0^{kt}
    \wh f_{m,L}(\eta L + x)\,dx\bigg).
  \]
From \eqref{t2.goal}, the fact that $f_L$ has period 1, the identity $L=p/q$, and the substitution $x=ks$, this gives
  \begin{align*}
  \lim_{n\to\infty}S_n(t) &= \frac3{4Lk}\int_0^{kt}
    \frac1q\sum_{\eta=1}^q f_L(\eta L+x)\,dx\\
  &= \int_0^t\frac3{4p}\sum_{\eta=1}^q f_L(\eta L + ks)\,ds.
  \end{align*}

\medskip

\noindent\textbf{Step 4.} Assume $L\in\bQ$, $a_n|\de_n| \to k\in(0,\infty)$, and $\de_n
<0$ for all $n$. As in Step 3, we have
  \[
  \lim_{n\to\infty}S_n(t)
    = \int_0^t\frac3{4p}\sum_{\eta=1}^q f_L(\eta L - ks)\,ds.
  \]
By Lemma \ref{lem1},
  \begin{align*}
  \lim_{n\to\infty}S_n(t)
    &= \int_0^t\frac3{4p}\sum_{\eta=1}^q
    f_L((q - \eta +1) L + ks)\,ds\\
  &= \int_0^t\frac3{4p}\sum_{\eta=1}^q f_L(\eta L + ks)\,ds.
  \end{align*}

\medskip

\noindent\textbf{Step 5.} We now prove Theorem \ref{t2}(i). From the discussion following the statement of Theorem \ref{t2}(i), we have that Theorem \ref{t2}(i) with $k=0$ is equivalent to Theorem \ref{C:main3}(i). Thus, we may assume $a_n|\de_n|\to k\in(0,\infty)$. Let $\{S_{n_m}\}$ be any subsequence of $\{S_n\}$. Recall from Step 1 that $\de_n\ne0$ for all $n\ge n_1$.  Choose a subsequence $\{S_{n_{m(j)}}\}$ of $\{S_{n_m}\}$ such that $\sgn(\de_{n_{m(j)}})$ does note depend on $j$. By Steps 3 and 4,
  \[
  \lim_{j\to\infty}S_{n_{m(j)}}(t)
    = \int_0^t\frac3{4p}\sum_{\eta=1}^q f_L(\eta L + ks)\,ds.
  \]
Since every subsequence has a subsequence converging to this limit, it follows that
  \[
  \lim_{n\to\infty}S_n(t)
    = \int_0^t\frac3{4p}\sum_{\eta=1}^q f_L(\eta L + ks)\,ds.
  \]
Note that $\eta L=\eta p/q$ and, since $p$ and $q$ are relatively prime,
  \[
  \{\eta p/q: 1 \le \eta \le q\}
    = \{j/q: 1 \le j \le q\}.
  \]
Thus,
  \[
  \lim_{n\to\infty}S_n(t)
    = \int_0^t\frac3{4p}\sum_{j=1}^q
    f_L\left({\frac jq + ks}\right)\,ds.
  \]
By \eqref{t2.pregoal}, this completes the proof of Theorem \ref{t2}(i).

\medskip

\noindent\textbf{Step 6.} We now prove Theorem \ref{t2}(ii). From Step 2, it suffices to consider $L\notin\bQ$. As in the proof of the equidistribution theorem,  the idea is to approximate the function $f_{m,L}$ by its truncated Fourier series.

Fix $m\in\bZ$ and let $\ep>0$ be arbitrary. Set
  \[
  F_N(x) =\sum_{k=-N} ^{N} c_k e^{2\pi i kx},
  \]
where
  \[
  c_k= \int_0^1 f_{m,L} (y) e^{2\pi i ky}dy.
  \]
Since $\|f_{m,L}\|_\infty\le8$ (see \cite[(2.23)]{BNS}), we have $|c_k|\le 8$.

The function $f_{m,L}$ is H\"older continuous of order $1/3$. Therefore, by Jackson's theorem,  the sequence $F_N$ converges uniformly on $[0,1]$ to $f_{m,L}$, and we may choose $N\in\bN$ such that for all $x\in [0,1]$,
  \[
  |F_N(x) - f_{m,L}(x) |  < \ep.
  \]
Recalling that $\{x\} =x-\flr{x}$, we then have for any fixed $t>0$,
  \begin{align*}
  \wt\be(m,n)
    &= \frac 1{a_n}\sum_{j=1}^{\flr{a_nt}} f_{m,L}(\{jL_n\})
    =\frac 1{a_n}\sum_{j=1}^{\flr{a_nt}} F_N(\{jL_n\}) + O(\ep)\\
  &= \frac 1{a_n}\sum_{k=-N}^N c_k\sum_{j=1}^{\flr{a_nt}}
    e^{2\pi ikjL_n} + O(\ep).
  \end{align*}
In the above and for the remainder of this proof, the coefficients implied by the big O notation depend only on $t$.

Note that for any integer $M\ge 1$ and for any complex number $\al$,
  \begin{equation}\label{eq4}
  \sum_{j=1}^M \al^j = \begin{cases}
      \frac{\al(1 - \al^M)}{1 - \al} &\text{if $\al\ne1$},\\
      M                              &\text{if $\al=1$}.
    \end{cases}
  \end{equation}
Set $\si_{k,n} = \sum_{j=1}^{\flr{a_nt}} e^{2\pi ikjL_n}$. Then, $\si_{k,n} =\flr{a_nt}$ if $kL_n \in \bZ$. If $kL_n \notin \bZ$, then
  \[
  |\si_{k,n}| = \left|{e^{2\pi ikL_n}
    \frac{1 - e^{2\pi ikL_n\flr{a_n t}}} {1 - e^{2\pi ikL_n}}
    }\right|
    \le \frac 2{|1 - e^{2\pi ikL_n}|}.
  \]
Since $L_n$ converges to $L$ which is irrational, there exists $\de>0$ and $n_0\in \bN$ such that for all $n\ge n_0$ and for all $k \in \{-N, \dots, N\}$, we have $|1- e^{2\pi i kL_n}| \ge \de$. Therefore, $|\si_{k,n}| \le 2/\de$ whenever $n\ge n_0$, $k \in \{-N, \dots, N\}$, and $kL_n\notin\bZ$.

Recall that $L_n =p_n/q_n$, where $p_n$ and $q_n$ are relatively prime numbers. Hence, $kL_n \in \bZ$ if and only if $q_n\mid k$. Therefore, we obtain
  \begin{align*}
  \wt\be(m,n) &= \frac 1{a_n}\sum_{k=-N}^{N} c_k\si_{k,n} + O(\ep)\\
  &= \frac 1{a_n}\sum_{\substack{k=-N\\ q_n\mid k}}^{N} c_k \flr{a_n t}
    + R_n + O(\ep),
  \end{align*}
where
  \[
  |R_n| = \bigg|\frac 1{a_n}\sum_{\substack{k=-N\\ q_n\nmid k}}^{N}
    c_k\si_{k,n}\bigg|
    \le \frac 2{a_n\de}\sum_{k=-N}^{N}|c_k| = O(Na_n ^{-1}\de^{-1}).
  \]
By (\ref{eq4}), 
  \[
  \frac 1{q_n} \sum_{j=1}^{q_n} (e^{2\pi ik/q_n})^j
    = \begin{cases}
      1 &\text{if $q_n \mid k$},\\
      0 &\text{if $q_n \nmid k$}.
  \end{cases}
  \]
As a consequence, we can write
  \begin{align*}
  \wt\be(m,n) &=  \frac{\flr{a_nt}}{a_n}\sum_{k=-N}^N
    c_k\frac 1{q_n}\sum_{j=1}^{q_n} e^{2\pi ikj/q_n}
    + O(Na_n ^{-1}\de^{-1}) + O(\ep)\\
  &= \frac{\flr{a_n t}}{a_nq_n}\sum_{j=1}^{q_n} F_N(j/q_n)
    + O(Na_n ^{-1}\de^{-1}) + O(\ep)\\
  &= \frac{\flr{a_n t}}{a_nq_n}\sum_{j=1}^{q_n} f_{m,L}(j/q_n)
    + O(Na_n ^{-1}\de^{-1}) + O(\ep).
  \end{align*}
In \cite{BNS}, it is shown that $q_n\to\infty$ when $L\notin\bQ$. Thus, letting $n$ tend to infinity and using the fact that $f_{m,L}$ is Riemann integrable on $[0,1]$ gives
  \[
  \limsup_{n\to\infty}\bigg|
    \wt\be(m,n) - t\int_0^1 f_{m,L}(x)\,dx
    \bigg| = O(\ep).
  \]
Since $\ep$ was arbitrary, and from \eqref{t2.goal} and \eqref{t2.pregoal}, this completes the proof.
\end{proof}

\noindent\textbf{Examples.} Here are some examples that were not covered by the results of \cite{BNS}. Suppose that $a_n =n^2$ and $b_n =(n+1)^2$. In this case $L_n \to 1$ and $a_n|\de_n|=|b_n-a_nL|=2n+1 \to \infty$. Therefore,
 \[
 \rho(t) = \frac 34 \int_0^1  f_1(x)\,dx.
 \] 
If $a_n =2n$ and $b_n =3n+1$, then $L_n \to 3/2$ and $a_n |\de_n|=|b_n-a_n L|=1$ for all $n$. Therefore,  
 \[
 \rho(t) = \frac 14 \left(
   f_{3/2}\left(\frac 12 + t\right) + f_{3/2}(t)
   \right).
 \]


\begin{thebibliography}{1}

\bibitem{BNS}
Krzysztof Burdzy, David Nualart, and Jason Swanson.
\newblock Joint convergence along different subsequences of the signed cubic
  variation of fractional {B}rownian motion.
\newblock Preprint, arxiv:1210.1560, October 2012.

\bibitem{BS}
Krzysztof Burdzy and Jason Swanson.
\newblock A change of variable formula with {I}t\^o correction term.
\newblock {\em Ann. Probab.}, 38(5):1817--1869, 2010.

\bibitem{NuPe}
I.~Nourdin and G.~Peccati.
\newblock {\em Normal Approximations with Malliavin Calculus: From Stein's
  Method to Universality}.
\newblock Cambridge University Press, 2012.

\bibitem{NR}
Ivan Nourdin and Anthony R{\'e}veillac.
\newblock Asymptotic behavior of weighted quadratic variations of fractional
  {B}rownian motion: the critical case {$H=1/4$}.
\newblock {\em Ann. Probab.}, 37(6):2200--2230, 2009.

\bibitem{NoReSw}
Ivan Nourdin, Anthony R{\'e}veillac, and Jason Swanson.
\newblock The weak {S}tratonovich integral with respect to fractional
  {B}rownian motion with {H}urst parameter {$1/6$}.
\newblock {\em Electron. J. Probab.}, 15:no. 70, 2117--2162, 2010.

\bibitem{NualartOrtiz2008}
D.~Nualart and S.~Ortiz-Latorre.
\newblock Central limit theorems for multiple stochastic integrals and
  {M}alliavin calculus.
\newblock {\em Stochastic Process. Appl.}, 118(4):614--628, 2008.

\bibitem{PeccatiTudor2005}
Giovanni Peccati and Ciprian~A. Tudor.
\newblock Gaussian limits for vector-valued multiple stochastic integrals.
\newblock In {\em S\'eminaire de {P}robabilit\'es {XXXVIII}}, volume 1857 of
  {\em Lecture Notes in Math.}, pages 247--262. Springer, Berlin, 2005.

\bibitem{Swanson2011a}
Jason Swanson.
\newblock The calculus of differentials for the weak {S}tratonovich integral.
\newblock Preprint, arxiv:1103.0341. To appear in the {\it Festschrift in Honor
  of David Nualart}, a volume to be published by Springer in the {\it
  Proceedings in Mathematics} Series., March 2011.

\bibitem{Wa}
Peter Walters.
\newblock {\em An introduction to ergodic theory}, volume~79 of {\em Graduate
  Texts in Mathematics}.
\newblock Springer-Verlag, New York, 1982.

\end{thebibliography}

\end{document}